\documentclass[12pt]{amsart}

\usepackage{amsfonts}
\usepackage{amssymb}
\usepackage{mathrsfs} 
\usepackage[letterpaper, left=2.5cm, right=2.5cm, top=2.5cm,
bottom=2.5cm,dvips]{geometry}
\usepackage{verbatim}
\usepackage[T1]{fontenc}
\usepackage{graphicx}
\usepackage{amsmath}
\usepackage{float}
\usepackage{polski}
\usepackage{diagbox,booktabs}
\usepackage[english]{babel}
\usepackage{pgf,tikz}
\usepackage{subcaption}
\newtheorem{theorem}{Theorem}
\theoremstyle{plain}

\newtheorem{conjecture}{Conjecture}
\newtheorem{corollary}{Corollary}

\newtheorem{proposition}{Proposition}

\numberwithin{equation}{section}

\def\ka #1{\mathscr{#1}}
\def\kal #1 #2{\mathscr{#1}^{#2}}

\usepackage{diagbox} 

\usepackage{array,multirow}

\usepackage{xcolor}

\usepackage[backref]{hyperref}
\hypersetup{
	colorlinks,
	linkcolor={blue!60!black},
	citecolor={red!60!black},
	urlcolor={red!60!black}
}

\DeclareMathOperator{\bp}{bp}

\DeclareMathOperator{\dist}{dist}
\DeclareMathOperator{\score}{score}

\begin{document}
	\title{Neighborly boxes and bipartite coverings; constructions and conjectures}
	
		\author{Jaros\l aw Grytczuk}
\address{Faculty of Mathematics and Information Science, Warsaw University
	of Technology, 00-662 Warsaw, Poland}
\email{jaroslaw.grytczuk@pw.edu.pl}

\author{Andrzej P. Kisielewicz}
\address{Wydzia{\l} Matematyki, Informatyki i Ekonometrii, Uniwersytet Zielonog\'orski, ul. Podg\'orna 50, 65-246 Zielona G\'ora, Poland}
\email{A.Kisielewicz@wmie.uz.zgora.pl}

\author{Krzysztof Przes\l awski}
\address{Wydzia{\l} Matematyki, Informatyki i Ekonometrii, Uniwersytet Zielonog\'orski, ul. Podg\'orna 50, 65-246 Zielona G\'ora, Poland}
\email{K.Przeslawski@wmie.uz.zgora.pl}

\begin{abstract}
Two axis-aligned boxes in $\mathbb{R}^d$ are \emph{$k$-neighborly} if their intersection has dimension at least $d-k$ and at most $d-1$. The maximum number of pairwise $k$-neighborly boxes in $\mathbb{R}^d$ is denoted by $n(k,d)$. It is known that $n(k,d)=\Theta(d^k)$, for fixed $1\leqslant k\leqslant d$, but exact formulas are known only in three cases: $k=1$, $k=d-1$, and $k=d$. In particular, the formula $n(1,d)=d+1$ is equivalent to the famous theorem of Graham and Pollak on bipartite partitions of cliques.

In this paper we are dealing with the case $k=2$. We give a new construction of $k$-neighborly \emph{codes} giving better lower bounds on $n(2,d)$. The construction is recursive in nature and uses a kind of ``algebra'' on \emph{lists} of ternary strings, which encode neighborly boxes in a familiar way. Moreover, we conjecture that our construction is optimal and gives an explicit formula for $n(2,d)$. This supposition is supported by some numerical experiments and some partial results on related open problems which are recalled.

\end{abstract}
	
	\maketitle
	
\section{Introduction}
We consider a combinatorial problem that can be stated in three different settings. The one we start with has a geometric flavor.

\subsection{Neighborly boxes in $\mathbb{R}^d$}
Consider a family of \emph{boxes} in $\mathbb{R}^d$, i.e., axis-parallel $d$-dimensional cuboids. Two boxes are said to be \emph{neighborly} if their intersection is a $(d-1)$-dimensional box. For instance, a pair of neighborly boxes in the plane is formed by two rectangles whose intersection is a non-trivial segment, while in the $3$-dimensional space, by two cuboids whose common part is a rectangle of positive area. How many pairwise neighborly boxes one may find in $\mathbb{R}^d$?

It is not hard to see that for the initial dimensions, $d=1$, $d=2$, and $d=3$, these numbers are equal respectively to $2$, $3$, and $4$. A general result asserting that the pattern continues was proved in 1985 by Zaks \cite{Zaks3}.
\begin{theorem}[Zaks, \cite{Zaks3}]\label{Theorem Zaks}
	The maximum number of pairwise neighborly boxes in $\mathbb{R}^d$ equals $d+1$.
\end{theorem}
The proof is based on a beautiful 1972 result of  Graham and Pollak \cite{GP} concerning bipartite partitions of complete graphs. We will explain this connection a bit later.

In \cite{Alon} Alon studied the following natural generalization of the problem of neighborly boxes. Let $1\leqslant k\leqslant d$ be a fixed integer. Two boxes in $\mathbb{R}^d$ are \emph{$k$-neighborly} if their intersection has dimension at least $d-k$ and at most $d-1$. Denote by $n(k,d)$, ($d\geqslant k$), the maximum number of pairwise $k$-neighborly boxes in $\mathbb{R}^d$. Clearly, Theorem \ref{Theorem Zaks} corresponds to the case $k=1$ and states that $n(1,d)=d+1$ for all $d\geqslant 1$.

Alon \cite{Alon} determined the asymptotic growth of the function $n(k,d)$ for all $k\geqslant 1$, namely $n(k,d)=\Theta (d^k)$, by proving the following inequalities for all $1\leqslant k\leqslant d$:
\begin{equation}
	\frac{1}{k^k}\cdot d^k\leqslant n(k,d)\leqslant \frac{2\cdot(2e)^k}{k^k}\cdot d^k.
\end{equation}
These bounds were recently improved in \cite{AlonGKP} and \cite{ChengWXY}. In particular, in \cite{AlonGKP}, the following lower bound for $n(k,d)$ was obtained:
\begin{equation}
	n(k,d)\geqslant (1-o(1))\frac{d^k}{k!}.
\end{equation}
Moreover, the following conjecture was posed in \cite{AlonGKP}.
\begin{conjecture}\label{Conjecture Limit}
	For every fixed integer $k\geqslant 1$, there exists a real number $\gamma_k$ such that
	\begin{equation}
		\lim_{d\rightarrow \infty}\frac{n(k,d)}{d^k}=\gamma_k.
	\end{equation}
\end{conjecture}
By Theorem \ref{Theorem Zaks}, $\gamma_1=1$, but for every $k\geqslant 2$ the conjecture is widely open. It is however tempting to guess that perhaps $\gamma_k=\frac{1}{k!}$ for every $k\geqslant 1$. Indeed, in \cite{AlonGKP} we made another supposition, which (if true) would imply this guess.

Let us expand the definition of $n(k,d)$ by adopting the convention that $n(0,d)=1$ for all $d\geqslant 1$.
\begin{conjecture}\label{Conjecture Pascal}
	For every $1\leqslant k\leqslant d$,
	\begin{equation}
		n(k,d)\leqslant n(k-1,d-1)+n(k,d-1).
	\end{equation}
\end{conjecture}

For instance, for $k=2$ this would give $n(2,d)\leqslant d^2/2+O(d)$ implying that $\gamma_2=1/2$. In general, by induction and the well known formula for the sum $1^k+2^k+\cdots +d^k$, one easily gets $n(k,d)\leqslant d^k/k!+O(d^{k-1})$, which shows that Conjecture \ref{Conjecture Pascal} implies Conjecture \ref{Conjecture Limit} with $\gamma_k=1/k!$.

Let us mention that exact formulas for $n(k,d)$ are known only in three cases, namely, $n(1,d)=d+1$ (Zaks' theorem), $n(d,d)=2^d$ (trivial), and $n(d-1,d)=3\cdot 2^{d-2}$ (obtained recently in \cite{AlonGKP}). This shows that Conjecture \ref{Conjecture Pascal} is true for the three corresponding cases, $k=1$, $k=d$, and $k=d-1$.

In the present paper we will give a new construction of families of pairwise $2$-neighborly boxes improving the recent lower bound from \cite{AlonGKP}. Furthermore, we conjecture that this construction is optimal and gives an exact formula for $n(2,d)$.

\subsection{Binary strings with jokers}
The problem of neighborly boxes in $\mathbb{R}^d$ can be encoded in a purely combinatorial setting using strings over alphabet with just three symbols. Let $J=[0,1]$ be the unit segment and $J_0=[0,1/2]$ and $J_1=[1/2,1]$, its left and right half, respectively. A \emph{normalized} box is a $d$-dimensional cuboid of the form $A=A_1\times\cdots \times A_d$, where $A_i\in \{J_0,J_1,J\}$ for all $i=1,2,\ldots, d$.

Suppose that we are given two normalized boxes, $A=A_1\times\cdots \times A_d$ and $B=B_1\times\cdots \times B_d$. If for some fixed coordinate $i$, we have $\{A_i,B_i\}=\{J_0,J_1\}$, then we say that $A$ and $B$ \emph{pass} each other in dimension $i$. Otherwise, we say that $A$ and $B$ \emph{overlap} in dimension $i$. Clearly, the intersection $A\cap B$ is a cuboid whose dimension equals exactly the number of dimensions in which $A$ and $B$ overlap. For instance, if $A$ and $B$ overlap in exactly $d-1$ dimensions, or the same, if they pass each other in exactly one dimension, then $A$ and $B$ are neighborly. In general, two normalized boxes are $k$-neighborly if and only if they pass in at least one and at most $k$ dimensions.

It is not hard to imagine that any family of boxes in $\mathbb{R}^d$ can be transformed to a family of normalized boxes preserving dimensions of all intersecting pairs. Therefore in investigating the function $n(k,d)$ one may only restrict to normalized boxes.

To further simplify the setting, let $S=\{0,1,\ast\}$ be an alphabet consisting of two binary digits and one special symbol called \emph{joker}. Let $S^d$ be the set of all strings of length $d$ over $S$. Clearly, a normalized box $A=A_1\times\cdots \times A_d$ can be identified with a string $u=u_1u_2\cdots u_d$ so that $u_i=0$ if $A_i=J_0$, $u_i=1$ if $A_i=J_1$, and $u_i=\ast$ if $A_i=J$.  

\begin{figure}[ht]
	\begin{center}
		
		\resizebox{14cm}{!}{			
			\includegraphics{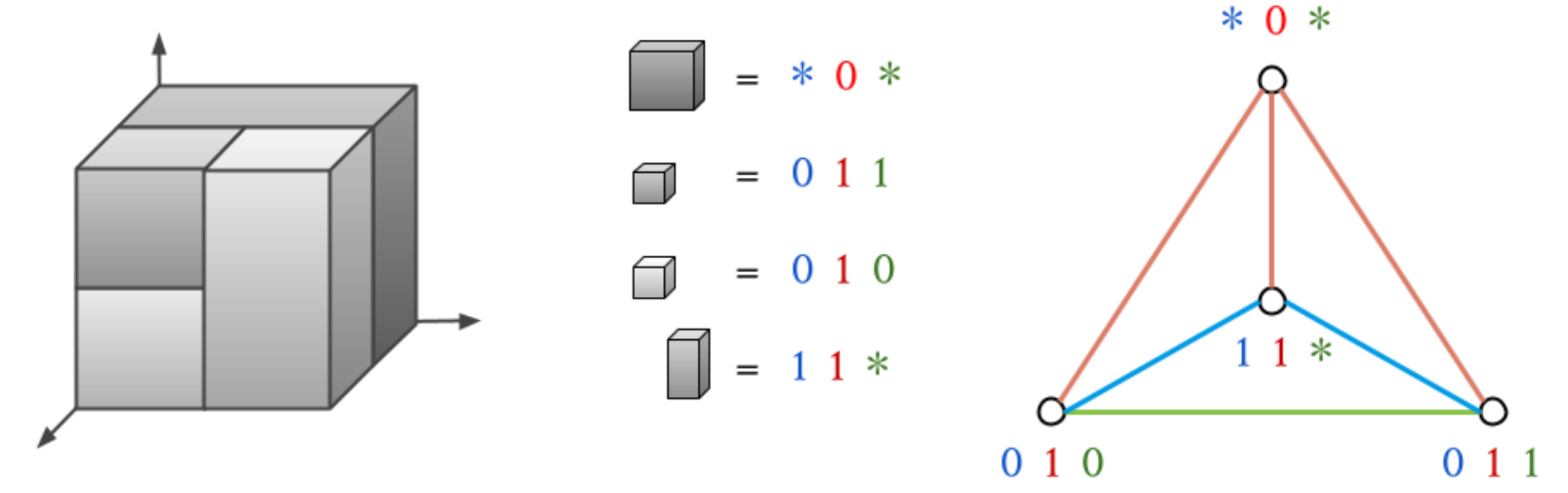}}
		
				\caption{Four pairwise neighborly boxes in $\mathbb{R}^3$, the corresponding neighborly code, and a bipartite clique partition.}
		\label{Boxes Strings Clique}
	\end{center}
\end{figure}

The \emph{distance} between two strings $u,v\in S^d$ is defined as the number of positions where they differ, but none of them is occupied by a joker. It is denoted by $\dist(u,v)$. More formally, if $u=u_1\cdots u_d$ and $v=v_1\cdots v_d$, then
$$
\dist(u,v)=|\{1\leqslant i\leqslant d: u_i\neq v_i \text{ and }u_i,v_i\in \{0,1\}\}|.
$$
Notice that the distance may be zero even if the two strings are not the same. For instance, if $u=0\ast1\ast$ and $v=\ast 1\ast 0$, then $\dist(u,v)=0$.

Clearly, $\dist(u,v)$ is exactly the number of dimensions in which the two corresponding boxes pass. Thus, by the above discussion, $n(k,d)$ is just the maximum number of strings in $S^d$ such that every two of them satisfy $1\leqslant \dist(u,v)\leqslant k$. In particular, Theorem \ref{Theorem Zaks} can be formulated in the following equivalent form.

\begin{theorem}[Zaks, \cite{Zaks3}]\label{Theorem Zaks Strings}
	The maximum number of strings in $S^d$, each two of them is at distance one, equals $d+1$.
\end{theorem}
 
 Let us call any family of strings in $S^d$ a \emph{$k$-neighborly code} if the restriction $1\leqslant \dist(u,v)\leqslant k$ holds for each pair of strings in the family. So, $n(k,d)$ is the maximum size of a $k$-neighborly code in $S^d$.
 
\subsection{Bipartite coverings of graphs}
Recall that a \emph{complete graph} (or a \emph{clique}) is a graph in which every pair of vertices is connected by an edge. A clique on $n$ vertices is denoted by $K_n$. A \emph{complete bipartite graph} (or a \emph{bipartite clique}) is a graph whose set of vertices can be split into two non-empty subsets, $X$ and $Y$, with no edges inside $X$ or $Y$, and all possible edges between $X$ and $Y$. A bipartite clique on sets of size $m$ and $n$ is denoted by $K_{m,n}$.

In 1972 Graham and Pollak \cite{GP} proved the following theorem.

\begin{theorem}[Graham and Pollak, \cite{GP}]\label{Theorem Graham-Pollak}
	The minimum number of complete bipartite graphs needed to partition the edges of a complete graph on $n$ vertices is $n-1$.
\end{theorem}

A beautiful proof of this result based on simple linear algebra was found by Tverberg \cite{Tverberg}. It is included in the famous ``Proofs from the Book'' by Aigner and Ziegler \cite{Aigner Ziegler}. Curiously, all know proofs of this result are more or less ``algebraic'' (cf. \cite{GP}, \cite{Peck}, \cite{Vishwanathan1}, \cite{Vishwnathan2}) and finding a purely combinatorial argument remains quite a challenge. There exist various generalizations of the Graham-Pollak theorem for which many challenging problems remain open (cf. \cite{Alon 2}, \cite{Cioaba}, \cite{van Lint}).

To see the connection with strings suppose that $B_1,\ldots, B_d$ is a family of bipartite cliques whose edges partition the set of edges of a clique $K_n$. For every vertex $u$ of the clique, assign a string $u=u_1\cdots u_d$ in $S^d$ defined as follows. If $X_i$ and $Y_i$ are the two bipartition subsets of $B_i$, then put $u_i=0$ if $u \in X_i$, $u_i=1$ if $u\in Y_i$, and $u_i=*$ if $u\notin X_i\cup Y_i$. Each pair of these strings is at distance one. Indeed, the edge joining two vertices $u$ and $v$ must belong to exactly one bipartite clique, say $B_j$. Then the corresponding strings $u_1\cdots u_d$ and $v_1\cdots v_d$ satisfy $\{u_j,v_j\}=\{0,1\}$ and this happens only at position $j$. Vice versa, having $1$-neighborly code of size $n$ in $S^d$ we may reverse the above process and assign to each of the $d$ dimensions the unique bipartite clique, thereby obtaining a desired partition of $K_n$. It follows that Theorems \ref{Theorem Zaks Strings} and \ref{Theorem Graham-Pollak}, are indeed equivalent. (See Figure \ref{Boxes Strings Clique} for an example of pairwise neighborly boxes encoded as strings and a bipartite clique partition of a complete graph.) 

Analogous argument works for arbitrary $k$-neighborly codes which correspond in the same way to bipartite clique \emph{$k$-coverings} of $K_n$, where each edge belongs to at least one and at most $k$ bipartite cliques of the covering.

Let us denote by $\bp_k(G)$ the least number of bipartite cliques needed in such $k$-covering of a graph $G$. For instance, by Theorem \ref{Theorem Graham-Pollak}, we have $\bp_1(K_n)=n-1$. By the above discussion we get the following statement formulated by Alon in \cite{Alon}.
\begin{proposition}[Alon, \cite{Alon}]\label{Proposition Alon}
	The number $n(k,d)$ is the largest positive integer $N$ such that $\bp_k(K_N)\leqslant d$.
\end{proposition}

\subsection{The main result}
In the present paper we focus entirely on the case $k=2$. We give a new construction of $2$-neighborly codes for all $d\geqslant 4$. This leads to a slight improvement of the lower bound for $n(2,d)$ obtained in \cite{AlonGKP}. Moreover, we suspect that this construction is optimal and gives a complete information on the function $n(2,d)$.

To state the result stemming from our construction, let us denote by $a(n)$, ($n\geqslant 1$), the $n$-th term of the following sequence of numbers:
\begin{equation*}
	2,\mathbf{3},\mathbf{3},4,5,\mathbf{6},\mathbf{6},7,8,9,10,11,\mathbf{12},\mathbf{12},13,14,15,16,17,18,19,20,21,22,23,\mathbf{24},\mathbf{24},25,....
\end{equation*}
The sequence can be described as a non-decreasing list of all positive integers, with the number $1$ missing, in which every number of the form $3\cdot 2^r$, ($r\geqslant 0$), is repeated twice.

Let us denote for convenience $b(d)=4+a(1)+\cdots+a(d-2)$, for $d\geqslant 3$, and $b(2)=4$. Our main result reads as follows.
\begin{theorem}\label{Theorem Main a(n)}
	For every $d\geqslant 2$, we have
	\begin{equation}
		n(2,d)\geqslant b(d).
	\end{equation}
\end{theorem}

As mentioned above, we conjecture that the lower bound from Theorem \ref{Theorem Main a(n)} is optimal.

\begin{conjecture}\label{Conjecture n(2,d)=b(d)} For each $d\geqslant 2$, $n(2,d)=b(d)$.
\end{conjecture}
The initial values of the function $b(d)$ are collected in Table \ref{Table b(d)}. The conjectured equality $n(2,d)=b(d)$ was computationally verified by Łuba \cite{Luba} up to $d=7$. Let us mention, however, that the currently best upper bound for $n(2,d)$, due to Huang and Sudakov \cite{HS}, is $n(2,d)\leqslant d^2+1$, which is roughly twice as big as the function $b(d)$.

\begin{table}[H]
	\begin{center}
		\begin{tabular} {|c|c|c||c|c|c|c||c|c|c|c|c|c|c||c|c|c|c|c|c|c|c|c|}
			\hline
			$d$ &2&3&4&5&6&7&8&9&10&11&12&13&14&15&16&17&18&19&20\\
			\hline
			$b(d)$ &4&6&9&12&16&21&27&33&40&48&57&67&78&90&102&115&129&144&160\\
			\hline
		\end{tabular}
	\end{center}
	
	\caption{Initial values of the function $b(d)$.}
	\label{Table b(d)}
\end{table}

Clearly, the function $b(d)$ can be explicitly determined. Let $d \geqslant 4$, and let $r\geqslant 0$ be the unique integer such that  
$$
3\cdot 2^r+r+1\leqslant d\leqslant
3\cdot 2^{r+1} +r+1.
$$
Then one may calculate that
$$
b(d)=  \frac {d^2}2 - \frac {d(2r+3)}2 + 6\cdot 2^r +\frac {(r+1)(r+2)} 2.
$$

By Proposition \ref{Proposition Alon} we may state Theorem \ref{Theorem Main a(n)} and Conjecture \ref{Conjecture n(2,d)=b(d)} in terms of bipartite coverings and the number $\bp_2(K_n)$.
\begin{theorem}\label{Theorem Main bp}
	For each $d\geqslant 2$, we have $\bp_2(K_{b(d)})\leqslant d$. 
\end{theorem}

To state our conjecture using the number $\bp_2(K_n)$, let us define the sequence of numbers $c(n)$, ($n\geqslant 2$), which is in some sense dual to $b(d)$. It can be described as the non-decreasing list of all positive integers such that the last occurrence of any number $d\geqslant 2$ is at $b(d)$ (see Table \ref{Table c(n)}). In particular, we have $c(b(d))=d$, for all $d\geqslant 2$. 
 
 \begin{table}[H]
 	\begin{center}
 		\begin{tabular} {|c|c|c|c|c|c|c|c|c|c|c|c|c|c|c|c|c|c|c|c|c|c|c|c|c|}
 			\hline
 			$n$ &\textbf{2}&3&\textbf{4}&5&\textbf{6}&7&8&\textbf{9}&10&11&\textbf{12}&13&14&15&\textbf{16}&17&18&19&20&\textbf{21}&22&23\\
 			\hline
 			$c(n)$ &\textbf{1}&2&\textbf{2}&3&\textbf{3}&4&4&\textbf{4}&5&5&\textbf{5}&6&6&6&\textbf{6}&7&7&7&7&\textbf{7}&8&8\\
 			\hline
 		\end{tabular}
 	\end{center}
 	
 	\caption{Initial values of the function $c(n)$.}
 	\label{Table c(n)}
 \end{table}
 
 By Proposition \ref{Proposition Alon} the following statement is equivalent to Conjecture \ref{Conjecture n(2,d)=b(d)}.
 \begin{conjecture}\label{Conjecture bp}
 	For every $n\geqslant 2$, we have $\bp_2(K_n)=c(n)$.
 \end{conjecture}
By the work of Łuba \cite{Luba} we know that the conjectured equality $\bp_2(K_n)=c(n)$ holds up to $n=27$. Thus the first occasion to disprove the conjecture is to find a bipartite $2$-covering of $K_{28}$ with $8$ bipartite cliques.

\section{Proof of the main result}
To prove Theorem \ref{Theorem Main a(n)} we will construct recursively a family of $2$-neighborly codes of sizes equal to $b(d)$ for each $d\geqslant 3$. This construction resembles a product constructions from \cite{Alon} and \cite{AlonGKP}, but there are some new features allowing to equip constructed codes with additional structural properties. Though a basic idea is not very complicated, to describe it in detail we need to build a kind of ``algebra'' on ordered lists of strings, or even on triples of lists. This setting may look a bit complicated at first glance, but we believe that it is interesting in its own and has some potential for future applications in constructions of $k$-neighborly codes for $k\geqslant 3$.

\subsection{Operations on lists of strings} 
We shall be concerned with ordered \emph{lists} of strings in $S^d$ and some operations on them. A list of strings will be written as $L=[v_i:i=1,\ldots,n]$, ($v_i\in S^d$), and we shall always assume that all strings in one list have the same length. We will also call the number $n$ of elements in the list $L$ the \emph{length} of $L$ and denote it by $|L|$. Notice that $|L|$ is not the same as the number of distinct strings in $L$, as one string may occur in many positions of the list.

If $A=[v_i:i=1,\ldots,n]$ and $B=[w_i:1,\ldots,n]$ are two lists of the same length $n$, then their \textit{pairing} is defined by
$$
A\ominus B= [v_iw_i\colon i=1,\ldots, n],
$$
where $v_iw_i$ is the usual concatenation of strings. So, the length of the pairing $A\ominus B$ equals to the length of each of the two components of the pairing. For example, if $A=[a,b,c]$ and $B=[x,y,z]$, then $A\ominus B=[ax, by, cz]$.

If $A$ and $B$ are two lists of not necessarily equal lengths $m$ and $n$, respectively, then we define their \textit{concatenation} by
$$
AB =[v_iw_j: i=1,\ldots, m, j=1,\ldots, n],
$$
where the linear order of elements in the new list is ''alphabetical'' with respect to pairs $(i.j)$. For example, if $A=[a,b,c]$ and $B=[x,y]$, then $AB=[ax, ay, bx, by, cx, cy]$. So, the length of the concatenation $AB$ is the product $mn$ of the lengths of the two components. 

If $A$ and $B$ are two lists containing strings of the same length, then the sum $A+B$ is a list of the elements of $A$ followed by the elements of $B$. If $A$ and $B$ are as in the previous example, then $A+B=[a,b,c,x,y]$. 

Lists $k\cdot A$, where $k$ runs over positive integers, are defined inductively: 
$$
1\cdot A=A; \quad (k+1)\cdot A= k\cdot A + A.
$$
We further assume that $A\cdot k$ is an alternative notation for $k\cdot A$. 

It can be convenient to write a list in block or array form. For example, if we write
$$
A=\begin{array}{cc}
	A_{11} & A_{12}\\
	A_{21} & A_{22}
\end{array},
$$
then we mean that the lists in the columns consist of strings of the same length and $A=A_{11}A_{12} + A_{21}A_{22}$. The general rule is to first concatenate the rows and then sum the resulting lists.

\subsection{Triples of lists; relations and operations}

We shall consider \emph{triples} of lists $T=(A, B, C)$ sharing the following properties:

\begin{enumerate}
	\item For any two strings $u,v \in A$, $\dist (u, v)\leqslant 1$;
	\item  $A$ and $B$ are of equal length and $A\ominus B$ is $2$-neighborly; 
	\item $C$ is a  $1$-neighborly sublist of $B$; moreover, for every element $u\in B$ and every $v\in C$, $\dist (u, v) \leqslant 1$. 
\end{enumerate}

For convenience, any triple $T=(A,B,C)$ satisfying properties (1)--(3) will be called \emph{nice}. 
We denote by $\alpha(T)$  the length of strings belonging to $A$. By $\beta(T)$, we mean the length of strings from $B$.  Then $\delta(T)=\alpha(T)+\beta(T)$ is the length of strings from $A\ominus B$. Moreover, the length $|A|$ of the list $A$ is denoted by $n(T)$. Then $n(T)=|B|=|A\ominus B|$.  Finally, $g(T)=|C|$. Since $A\ominus B$ is 2-neighborly, it follows that $n(T)\leqslant n(2, \delta(T))$. 

Two triples $T=(A, B, C)$ and $T'=(A', B', C')$ are \textit{concordant} if $\alpha(T)=\alpha(T')$ and for every $u, v$ from  $A+A'$, $\dist (u, v) \leqslant 1$.

For a pair of concordant triples $T,T'$, we may define their \textit{compound} $ T\otimes  T'=(A'', B'', C'')$, where 
\begin{equation}
	\label{konstrukcja}
	A''=\begin{array}{ll}[0]&A\\{[}0]&A' \\ {[}1]&[*^{\alpha(T)}]\cdot |CC'|\end{array}, \quad B''=\begin{array}{lll} [0] &B &[*^{\beta(T')}]\\ {[}1]&[*^{\beta(T)}]&B'\\{[}*]&C &C'\end{array}, \quad C'' =  \begin{array}{lll}[0] &C & [*^{\beta(T')}]\\{[}1]&[*^{\beta(T)}] &C'\end{array}.
\end{equation}
Let us remark that since $|CC'|=|C||C'|$, the third row of the expression for  $A''$ can be written alternatively: $g(T)g(T')\cdot {[}1][*^{\alpha(T)}]$. 

According to our definition, the pairing $A''\ominus B''$ can be expressed as follows:
\begin{equation}
	\label{agregat}
	A''\ominus B''=\begin{array}{llclll}[0]&A& \ominus &[0] &B &[*^{\beta(T')}]\\{[}0]&A' & \ominus & {[}1]&[*^{\beta(T)}]&B'\\{[}1]&[*^{\alpha(T)}]\cdot |CC'| & \ominus & {[}*]&C &C'\end{array}
\end{equation}
Regarding the order in which operations are performed in rows, we assume that concatenation precedes pairing. It is clear that $A''\ominus B''$ is 2-neighborly. Now, the following proposition is rather obvious.

\begin{proposition}
	\label{numery}
	If triples $T=(A, B, C)$ and $T'=(A', B', C')$ are nice and concordant, then $T''=T\otimes T'$ is nice as well. Moreover,
	\begin{enumerate}
		\item $\alpha(T'')=\alpha(T)+1=\alpha(T')+1$,
		\item $\beta(T'')=\beta(T)+\beta(T')+1$,
		\item $\delta(T'')=\delta(T)+\delta(T')-\alpha(T)+2$, 
		\item $g(T'')=g(T)+g(T')$,
		\item $n(T'')=n(T)+n(T')+g(T)g(T')$.
	\end{enumerate}
\end{proposition}

Two triples  $T=(A, B, C)$ and  $T'=(A', B', C')$ are \textit{congruent} if they are concordant,  $\beta (T)=\beta(T')$, $n(T)=n(T')$ and $g(T)=g(T')$. Clearly, then we also have $\alpha(T)=\alpha(T')$ and $\delta(T)=\delta(T')$ for congruent triples. As an immediate consequence of our definitions and the preceding proposition one has:
\begin{proposition}
	\label{mod}
	If $S,S'$ and $T,T'$ are two pairs of concordant (congruent) triples, then $S\otimes T$, $S'\otimes T'$ is a pair of concordant (congruent) triples. Also, $S\otimes T$ and $T\otimes S$ are congruent. 
\end{proposition}

Let us remind that by Theorem \ref{Theorem Zaks Strings}, $n(1,d)=d+1$. Therefore, if $T=(A,B,C)$ is  a triple such that $g(T)=\beta(T)+1$, then $C$ is the maximum length $1$-neighborly sublist of $B$. In fact, it has the maximum length among all $1$-neighborly lists of strings of length $\beta(T)$. This observation explains the meaning of the following assertion.
\begin{proposition}
	Let $T$ and $T'$ be concordant triples such that $g(T) =\beta(T)+1$ and  $g(T') =\beta(T')+1$, then $g(T'')=\beta(T'')+1$, where  $T''=T\otimes T'$.  
\end{proposition}
The proof is an immediate consequence of Proposition \ref{numery}.

\subsection{Sequences of triples of lists}

For a triple $T$, we can define by induction the sequence $T_k$, ($k\geqslant0$), as follows:
\begin{equation}
	\label{ind}
	T_0=T, \quad T_{k+1}= T_{k} \otimes T_{k}, (k\geqslant0). 
\end{equation}

The following proposition summarizes numerical properties of this sequence.

\begin{proposition} 
	\label{numprop}
	Let $T$ be a triple and let $T_k$, ($k\geqslant0$), be a sequence of triples  given by (\ref{ind}). Then
	\begin{enumerate}
		\item $\alpha(T_k)=\alpha(T)+k$,
		\item $\beta(T_k)=2^k\beta(T)+2^k-1$,
		\item $\delta(T_k)=2^k\beta(T) +\alpha(T) +2^k +k-1 = 2^k\delta(T) - (2^k-1)\alpha(T) +2^k+k -1$, 
		\item $g(T_k)=2^k g(T)$,
		\item $n(T_k)=2^k n(T) +\frac{4^k-2^k}2 g(T)^2$.
	\end{enumerate}
	Moreover, if $g(T)=\beta(T)+1$, then $g(T_k)=\beta(T_k)+1$ for every non-negative integer $k$. 
\end{proposition} 
The proof is again an easy application of Proposition \ref{numery}.

Our next goal is to define two sequences of triples as described in (\ref{ind}). Each of these sequences is determined by the appropriate choice of the starting triple $T_0$. 

Let $H=[00,01,1*]$. Then the concatenation $F_4=HH$ is $2$-neighborly. It consists of strings of length $4$ and $|F_4|=9$. Remind that $n(2,4)=9$. Thus $F_4$ is of the maximum length among $2$-neighborly lists containing strings of length $4$. Clearly, $F_4=A^\flat\ominus B^\flat$, where $A^\flat=3\cdot[00]+3\cdot[01]+3\cdot[1*]$ and $B^\flat=3\cdot H$.   Moreover, $H$ is a $1$-complementary sublist of $B^\flat$. Therefore, if we set $C^\flat=H$, then  the triple $T^\flat=(A^\flat,B^\flat, C^\flat)$ satisfies conditions $1-3$. Now, according to (\ref{ind}), we can define an infinite sequence of triples $T ^\flat_k=(A_k^\flat, B_k^\flat, C^\flat_k)$, ($k\geqslant 0$),  assuming $T_0=T^\flat$.

Let $H$ be as already defined, and $G=[0,1]$. Let $F_7$ be a list containing strings of length $7$,  defined as follows:
$$
F_7=\begin{array}{lllcc}
	{[}0] & G&{[}0] &H &[**]\\
	{[}0] & G& {[}1]& [**] & H\\
	{[}1] & [*] &{[}*] & H & H
\end{array}
$$
$F_7$ is $2$-neighborly and has length $21$. It appears that $F_7$ has the maximum length among  $2$-neighborly lists containing strings of length $7$; that is, $|F_7|=n(2,7)$. Again, we can extract a triple $T^\sharp=(A ^\sharp, B^\sharp, C^\sharp)$
from $F_7$:
\begin{align*}
A^\sharp &= 2\cdot(3\cdot[00]+3\cdot[01]) +9\cdot[1*]\\
B^\sharp &= 2\cdot[0]H[**] + 2\cdot [1][**]H+[*]HH\\
C^\sharp &=[0]H[**]+[1][**]H,
\end{align*}
which is clearly a nice triple. Now, according to (\ref{ind}) we define a sequence $T^\sharp_k=(A_k^\sharp,B_k^\sharp,C_k^\sharp)$, ($k\geqslant0$), assuming $T_0=T^\sharp$.

Proposition \ref{numprop} applied to $T^\flat$, $T^\sharp$ gives us:
\begin{proposition}
	\label{wartosci}
	For every integer $k\geqslant0$,
	$$
	\begin{array}{lclcl}
		\alpha^\flat_k &=& \alpha(T^\flat_k) &=& k+2,\\
		\beta^\flat_k  &=&  \beta(T^\flat_k)  &=& 3\cdot 2^k-1,\\
		\delta^\flat_k &=& \delta(T^\flat_k)  &=& 3\cdot 2^k +k+1,\\ 
		g^\flat_k    &=& g(T^\flat_k)            &=& 3\cdot 2^k,\\
		n^\flat_k    &=& n(T^\flat_k)            &=& \frac 9{2}\cdot 4^k +\frac 9{2}\cdot 2^k,\\
	\end{array}  
	$$
	and
	$$
	\begin{array}{lclcl}
		\alpha^\sharp_k &=& \alpha(T^\sharp_k) &=& k+2,\\
		\beta^\sharp_k  &=&  \beta(T^\sharp_k)  &=& 6\cdot 2^k-1,\\
		\delta^\sharp_k &=& \delta(T^\sharp_k)  &=& 6\cdot 2^k +k+1,\\ 
		g^\sharp_k    &=& g(T^\sharp_k)            &=& 6\cdot 2^k,\\
		n^\sharp_k    &=& n(T^\sharp_k)            &=& 18\cdot 4^k +3\cdot 2^k.\\
	\end{array}  
	$$
\end{proposition}

\subsection{Sequences of sets of mutually concordant triples}

If $\ka T$ is a set of pairwise concordant triples, then, by Proposition \ref{mod}, $\ka T \otimes \ka T=\{S\otimes T \colon S, T\in \ka T\}$ is also a set of pairwise concordant triples. Therefore, we can define a sequence $\ka T_k$, ($k\geqslant 0$), by induction:
\begin{equation}
	\label{iteracja}
	\ka T_0 =\ka T, \quad \ka T_{k+1}=\ka T_{k}\otimes \ka T_{k}.
\end{equation}

This sequence is fully specified by $\ka T$. We will examine the properties of such a sequence for an appropriately selected $\ka T$.  Our $\ka T$ will consist of four triples $ T^\flat,  T^\dag, T^\ddag, T^\sharp$. Two of them  are already defined. It remains to specify $T^\dag$ and $T^\ddag$.    

Let $L=[000,001,01*, 1**]$. Let $F_5$ be a list o string of length $5$  defined similarly as $F_4$; that is, $F_5=HL$. Then $F_5$ is $2$-neighborly and $|F_5|=n(2,5)=12$. Let us set $A^\dag=4\cdot[00] +4\cdot[01]+ 4\cdot[1*]$ and $B^\dag= 3\cdot L$. Clearly, $F_5=A^\dag\ominus B^\dag$. Moreover, $C^\dag = L$ is a sublist of $B^\dag$ which is $1$-neighborly. It easily seen that the triple $T^\dag=(A^\dag, B^\dag, C^\dag)$ is nice. 

As it concerns $T^\ddag$ it is defined in a similar manner as $T^\sharp$.  Let us set
$$
F_6=    \begin{array}{rrrcc}
	[0] & G&[0] & G&[**]\\
	{[}0] & G& [1]& [*]& H\\
	{[}1] & [*]&[*] & G & H.
\end{array}
$$
Then $F_6$ is $2$-neighborly and $|F_6|=n(2,6)=16$. We can extract from $F_6$ the following lists:
\begin{align*}
	A^\ddag &=2\cdot[00]+2\cdot[01]+3\cdot[00]+3\cdot[01]+6\cdot[1*]\\
	B^\ddag &=2 [0]G[**]+2[1][*]H+ [*]GH\\
	C^\ddag &=[0]G[**]+[1][*]H.
\end{align*}
Obviously, $F_6=A^\ddag\ominus B^\ddag$ and $T^\ddag=(A^\ddag, B^\ddag, C^\ddag)$ is nice. 

Observe that for every non-negative integer $k$, $T_k^\flat$ and $T_k^\sharp$ belong to $\ka T_k$. Moreover, we claim that the following property holds.

\begin{proposition}
	For every $S \in \ka T_k$, 
	\begin{equation}
		\label{ramka}
		\delta^\flat_k\leqslant \delta(S)\leqslant\delta^\sharp_k.     
	\end{equation}
\end{proposition}
\begin{proof}
To prove the assertion we use a simple induction argument. Obviously the inequalities hold true for $k=0$. Suppose that $T''\in \ka T_{k+1}$. Then $T''=T\otimes T'$ for some $T, T'\in \ka T_k$. By the induction hypothesis, 
$$
\delta^\flat_k\leqslant \delta(T)\leqslant\delta^\sharp_k, \quad\text{and}\quad  \delta^\flat_k\leqslant \delta(T')\leqslant\delta^\sharp_k. 
$$
Combining these two inequalities and applying Proposition \ref{numery}, yields
$$
2\delta^\flat_k\leqslant \delta(T'') +\alpha(T)-2\leqslant 2\delta^\sharp_k. 
$$
Since all the numbers $\alpha(T)$, $\alpha^\flat_k$ and $\alpha^\sharp_k$ are equal, we easily conclude, by subtracting $\alpha(T)-2$, that 
$$
\delta^\flat_{k+1} \leqslant \delta(T'') \leqslant\delta^{\sharp}_{k+1},
$$ 
which completes the proof.	
\end{proof}

Let us remark that by Proposition \ref{wartosci}, for every $k\geqslant0$,
$$
\delta_{k+1}^\flat=\delta_k^\sharp +1. 
$$
Consequently, the family consisting of all ranges of integers $I_{k}=\{ \delta_{k}^\flat,  \delta_{k}^\flat +1, \ldots,  \delta_{k}^\sharp\}$, ($k\geqslant0$), is a partition of the unbounded range $\{4,5,6,\ldots\}$.   
\begin{proposition} 
	\label{na}
	Let $\ka T=\{T^\flat, T^\dag, T^\ddag, T^\sharp\}$ and let $\ka T_k$, ($k\geqslant0$), be the sequence defined by (\ref{iteracja}). Then the mapping $T\mapsto\delta(T)$ sends every $\ka T_k$ `onto' $I_k$.   
\end{proposition}
\begin{proof}
Again, we proceed by induction. For $k=0$, our proposition is clearly true. If it is true for some $k\geqslant1$, then, since $\delta(T\otimes T')=\delta(T)+\delta(T')-\alpha^\flat_k+2$, for every $T, T' \in \ka T_k$, it suffices to show that the mapping  $I_k\times I_k \ni (\delta, \delta')\mapsto \delta+\delta'-\alpha^\flat_k+2\in I_{k+1}$ is `onto', which is rather obvious in the light of Proposition \ref{wartosci}. 
\end{proof}

\begin{proposition}
	\label{wyrazenia}
	Let $\ka T_k$, ($k\geqslant0$), be the same sequence as in Proposition \ref{na}. Let $T\in \ka T_k$ and $d=\delta(T)$. Then 
	$$
	n(T)=n_k^\flat+g_k^\flat(d-\delta^\flat_k) +\frac{(d-\delta^\flat_k)(d-\delta^\flat_k-1)}2
	$$
	and
	$$
	g(T)= g_k^\flat+ d-\delta^\flat_k=d-k-1.
	$$
\end{proposition}
\begin{proof}
	The proof is by induction. The case $k=0$ results from simple calculations for both functions. Let $T''\in \ka T_{k+1}$. Then $T''=T\otimes T'$ for some $T, T'\in \ka T_k$.  Let $d^i=\delta(T^i)$, where $i$ is  empty, prime or the double prime symbol. By Proposition \ref{numery} and the induction hypothesis, in the case of $g$ we have
	\begin{eqnarray*}
		g(T'') & = &g(T)+ g(T')\\
		& = & g_k^\flat+ d-\delta^\flat_k + g_k^\flat+ d'-\delta^\flat_k\\
		& = & 2g_k^\flat + (d+d' - \alpha^\flat_k+2) - (2\delta^\flat_k-\alpha^\flat_k+2)\\
		& = & g_{k+1}^\flat +d''-\delta^\flat_{k+1}, 
	\end{eqnarray*}
	which completes the proof of the expression for $g$. As for $n$, we have
	\begin{eqnarray*}
		n(T'') & = & n(T)+n(T')+g(T)g(T') \\
		& = & 	n_k^\flat+g_k^\flat(d-\delta^\flat_k) +\frac{(d-\delta^\flat_k)(d-\delta^\flat_k-1)}2+n_k^\flat+g_k^\flat(d'-\delta^\flat_k) +\frac{(d'-\delta^\flat_k)(d'-\delta^\flat_k-1)}2\\ 
		& + & (g_k^\flat+ d-\delta^\flat_k)( g_k^\flat+ d'-\delta^\flat_k ) \\
		& = & (2n_k^\flat+(g_k^\flat)^2) + 2g_k^\flat(d+d'-2\delta_k^\flat)\\
		& + & \left(\frac{(d-\delta^\flat_k)(d-\delta^\flat_k-1)}2+ \frac{(d'-\delta^\flat_k)(d'-\delta^\flat_k-1)}2 +(d-\delta^\flat_k)(d'-\delta^\flat_k))\right)
	\end{eqnarray*}
	The three summands of the last expression are equal to corresponding summands of the following 
	$$
	n(T'')= n_{k+1}^\flat +  g_{k+1}^\flat(d''-\delta_{k+1}^\flat) +  \frac{(d''-\delta^\flat_{k+1})(d''-\delta^\flat_{k+1}-1)}2,
	$$
	Our proof is complete.
\end{proof}

Since the expressions on $n(T)$ and $g(T)$ depend only on $d$, and the mapping $S\mapsto \alpha(S)$ is constant on each $\ka T_k$, we get the following statement.

\begin{corollary}
	\label{kong}
	Let $k\geqslant0$ and let $d$ be an integer such that $\delta^\flat_k\leqslant d\leqslant \delta^\sharp_k$. Then all the triples $T\in \ka T_k$ 
	satisfying the equation $d=\delta(T)$ are mutually congruent.
\end{corollary}

Theorem \ref{Theorem Main a(n)} is now an immediate consequence of Propositions \ref{wartosci}, \ref{na}, \ref{wyrazenia} and the fact that the family of ranges $I_k$, ($k\geqslant0$), is a partition of the range $\{4,5,6,\ldots\}$.  

\section{Further properties of sequences of triple sets}
\subsection{Elementary decompositions}

We proceed to examine the sequence $\ka T_k$, ($k\geqslant0$), where $\ka T=\{T^\flat, T^\dag, T^\ddag, T^\sharp\}$.

Let $L_k$ be any of the lists $[T^\varepsilon\in \ka T\colon \varepsilon \in \{0,1\}^k]$. Clearly, there are $4^{2^k}$ such lists.  Let $\varepsilon, \varepsilon' \in \{0,1\}^k$ differ only in the last place. Let us set $\varepsilon''=\varepsilon|\{1,2,\ldots, k-1\}$ and $T^{\varepsilon''}=T^{\varepsilon}\otimes T^{\varepsilon'}$. We can arrange all triples $T^{\varepsilon''}$ into a new list $L_{k-1}=[T^{\varepsilon}\in \ka T_{1}\colon \varepsilon\in \{0,1\}^{k-1}]$.  Continuing in this manner, we arrive to the list $L_1=[T^0, T^1]$,  whose elements belong to $\ka T_{k-1}$. Eventually, we end up with the triple $T=T^0\otimes T^1$ belonging to $\ka T_k$. Since $T$ is uniquely determined by $L_k$, we can adopt the following notation: $T=\bigotimes L_k$.  The list $L_k$ is said to be an \textit{elementary decomposition} of $T$.  

In the light of the definition of $\ka T_k$, the following proposition is obvious.
\begin{proposition}
	Let $\ka T_k$, ($k\geqslant0$), be as in Proposition \ref{na}. Then $T\in \ka T_k$, ($k\geqslant 1$), if and only if $T=\bigotimes L_k$ for some $L_k=[T^\varepsilon \in \ka T\colon \varepsilon \in \{0,1\}^k]$.  
\end{proposition}

If $T=\bigotimes L_k$, then by Proposition \ref{numery} and an easy induction argument we get
$$
\beta(T)= 2^k-1 +\sum_{\varepsilon \in \{0,1\}^k} \beta(T^\varepsilon). 
$$
Let  us remind that 
$$
\beta(T^\flat)=2, \quad \beta(T^\dag)=  3, \quad \beta (T^\ddag)=4, \quad \text{and} \quad \beta (T^\sharp)=5.
$$
Let  $p$ be the number of occurences of $T^\flat$ in $L_k$, $q$ be the number of occurences of $T^\dag$, $r$ that of $T^\ddag$ and $s$ that of $T^\sharp$. Then 
$$
\beta(T)-2^k+1=2p+3q+4r+5s
$$
Let $d=\delta(T)$.  Since $T$ and $T_k^\flat$ being elements of $\ka T_k$ are concordant, $\alpha(T)=\alpha(T_k^\flat)=\alpha_k^\flat$. By the  definition of $\delta(T)$ and Proposition \ref{wartosci}, 
$$
d-2^k-k-1=2p+3q+4r+5s. 
$$
Therefore, we easily conclude with the following statement.

\begin{proposition}
	Let the elements of  a list $L$ belong to $\ka T=\{T^\flat, T^\dag, T^\ddag, T^\sharp\}$. Let $p$, $q$, $r$, $s$ be the numbers of occurrences of $T^\flat, T^\dag, T^\ddag, T^\sharp$ in $L$, respectively. If there is an integer $k \geqslant 0$ such that 
	$$
	p + q + r + s  =2^k,  \leqno{\rm (I)}
	$$
	and $d$ is defined by
	$$
	d= 2p + 3q + 4r + 5s +2^k+ k+1, \leqno{\rm (II)}
	$$
	then $L$ is  an elementary decomposition of some $T\in \ka T_k$ with $\delta(T)=d$.
\end{proposition} 

Let us fix $k$, set $r=s=0$, and find all possible solutions of the system (I-II) with unknown non-negative integers $p, q, d$. It is easy to see that we can represent the solutions in such a way that $d$ and $p$ are variables dependent on $q$:
\begin{align*}
	p &=2^k-q\\
	d &=q+ 3\cdot 2^k+k+1.
\end{align*}
Clearly, $q$ can vary in the range $\{0,1,\ldots, 2^k\}$.  Therefore,  $d$ can take as values all numbers from the range $\{\rho, \rho+1, \ldots, \rho +2^k\}$, where $\rho= 3\cdot 2^k+k+1$, and only those numbers.  

Similarly, if $p=s=0$, then
$$
d=r+4\cdot 2^k+k+1.
$$
And the range of $d$ is $\{\rho', \rho'+1, \ldots, \rho' +2^k\}$, where $\rho'= 4\cdot 2^k+k+1$. 

Finally, if $p=q=0$, then 
$$
d=s+5\cdot 2^k+k+1. 
$$

Now, the range of $d$ is $\{\rho'', \rho''+1, \ldots, \rho'' +2^k\}$, where $\rho''= 5\cdot 2^k+k+1$.

Observe that the consecutive ranges $\{\rho, \rho+1, \ldots, \rho +2^k\}$, $\{\rho', \rho'+1, \ldots, \rho' +2^k\}$, and $\{\rho'', \rho''+1, \ldots, \rho'' +2^k\}$ are adjacent. Each pair of subsequent ranges has one element in common.

We summarize these observations in the following two propositions.

\begin{proposition} Let $d\geqslant 4$. Then there is a unique non-negative integer $k$ so that $d$ satisfies one  of the inequalities: 
	\begin{itemize}
		\item[\textrm{(a)}]  $\rho \leqslant d < \rho'$,  
		\item[\textrm{(a')}] $\rho'\leqslant d < \rho''$,
		\item[\textrm{(a'')}] $\rho''\leqslant d\leqslant \rho'''$,  
	\end{itemize}
	where $\rho=3\cdot 2^k+k+1$, $\rho'= \rho+2^k$, $\rho''=\rho'+2^k$, $\rho'''=\rho''+2^k$ .
	
	If {\rm (a)} holds true, then  $L=(\rho'-d)\cdot [T^\flat]+(d-\rho)[T^\dag]$ is an elementary decomposition of some $T\in \ka T_k$ such that $d=\delta(T)$.
	
	If {\rm (a')} holds true, then  $L=(\rho''-d)\cdot [T^\dag]+(d-\rho')[T^\ddag]$ is an elementary decomposition of some $T\in \ka T_k$ such that $d=\delta(T)$.
	
	If {\rm (a'')} holds true, then  $L=(\rho'''-d)\cdot [T^\ddag]+(d-\rho'')[T^\sharp]$ is an elementary decomposition of some $T\in \ka T_k$ such that $d=\delta(T)$.
\end{proposition}

This proposition shows that for each $d\geqslant 4$ we may find a triple $T$ for which $d=\delta(T)$ and  whose elementary decomposition includes only two consecutive triples amongst $T^\flat$, $T^\dag$, $T^\ddag$, $T^\sharp$. 

According to (\ref{ind}), let us define two infinite sequences of triples, $T ^\dag_k$ and $T ^\ddag_k$, ($k\geqslant0$), whose initial triples are  $T^\dag_0=T^\dag$ and $T^\ddag_0=T^\ddag$, respectively. According to (\ref{iteracja}), let us define three sequences $\ka T^{l}_k$, $\ka T^{m}_k$, $\ka T^{t}_k$, ($k\geqslant0$), whose initial sets are as follows: $\ka T_0^l=\{T^\flat, T^\dag\}$, $\ka T_0^m=\{T^\dag, T^\ddag\}$,  $\ka T_0^t=\{T^\ddag, T^\sharp\}$. 

\begin{proposition}
	Let $\ka T_k$, ($k\geqslant0$), be as in Proposition \ref{na}. For every $k\geqslant0$, and every $T\in  \ka T_k$ there is $T'\in\ka T^l_k\cup \ka T^m_k\cup \ka T^t_k$ which is equivalent to $T$; that is $\delta(T)=\delta(T')$. Moreover,
	$$
	\ka T^l\cap \ka T^m=\{T^\dag\} \quad \text{and} \quad \ka T^m\cap \ka T^t=\{T^\ddag\}.
	$$ 
\end{proposition} 

\subsection{A distribution theorem}

Suppose that a finite set of concordant triples $\ka T$  and a positive integer  $k$ are given.  Let $\ka T_k$ be the $k$-th member of the sequence given by (\ref{iteracja})  and $T=(A,B,C)\in \ka T_k$. We will be dealing with the distribution of jokers $*$, or, essentially equivalently, the symbols $0,1$ in the strings belonging to $A\ominus B$.  

Let $L$ be a list of strings belonging to some $S^d$. For each string $w\in L$, let us take the total number of occurrences of the symbols $0$ or $1$ in this string. Let $\mu(L)$ and $M(L)$ be the minimum and the maximum of these numbers, respectively. 

For every triple $T=(A,B,C)$, we set $\mu(T)=\mu(A\ominus B)$, $M(T)=M(A\ominus B)$, $\kappa(T)=\mu(C)$ and  $K(T)=M(C)$. If $\ka S$ is a a non-empty finite set of triplets, then $\mu(\ka S)= \min\{\mu(T)\colon T \in \ka S\}$,   $M(\ka S)= \max\{M(T)\colon T \in \ka S\}$, $\kappa(\ka S)=\min\{\kappa(T)\colon T\in \ka S\}$ and $K(\ka S)= \max\{K(T)\colon T \in \ka S\}$. 

\begin{theorem}  Let $\ka T$ be a non-empty set of mutually concordant triples. Let the sequence $\ka T_k$, ($k\geqslant0$), be as defined in (\ref{iteracja}). Let $\mu_k=\mu(\ka T_k)$, $M_k=M(\ka T_k)$, $\kappa_0=\kappa(\ka T)$  and $K_0=K(\ka T)$. Then
	$$
	\min\{2\kappa_0 - 1,\mu_0\}+2k \leqslant \mu_k\leqslant M_k\leqslant \max\{2K_0-1, M_0\} +2k.
	$$
\end{theorem}
\begin{proof}
Let us set $\kappa_k =\kappa(\ka T_k)$ and $K_k=K(\ka T_k)$. First we prove by induction that 
$$
\kappa_0+k\leqslant \kappa_k\leqslant K_k\leqslant K_0+k.
$$

Let $T''=(A'', B'', C'')$ be an element of $\ka T_k$, where $k\geqslant 1$. Then $T''=T\otimes T'$, where $T=(A,B,C)$ and  $T'=(A',B',C')$ are elements of $\ka  T_{k-1}$. According to (\ref{konstrukcja}), 
$$
\kappa(T'')=\min\{\kappa(T)+1, \kappa(T') +1\}.
$$ 
By the induction, both  $\kappa(T)+1$, $\kappa(T') +1$ are not smaller than $\kappa_0+k$. Therefore, $\kappa_k\geqslant \kappa_0+k$. The inequality for $K_k$ is proved using the same method, with some obvious modifications. 

Now, by (\ref{agregat}),  
$$
\mu(T'')=\min\{\mu(T)+2, \mu (T')+2, \kappa(T)+\kappa(T')+1\}.
$$ 
Therefore, by the preceding part and the induction hypothesis,
$$
\mu(T'')\geqslant \min\{ \min\{2\kappa_0 - 1,\mu_0\}+2k, (2\kappa_0 - 1)+2k\}\geqslant \min\{2\kappa_0 - 1,\mu_0\}+2k,
$$
which readily implies the left-hand side inequality. The right-had side is proved in a similar manner.
\end{proof}

In the case of our particular interest $\ka T=\{\ka T^\flat, \ka T^ \dag, \ka T^\ddag, \ka T^\sharp\}$, we have $\mu_0=2$, $M_0=5$, $\kappa_0=1$, $K_0=3$. Therefore, we have the following corollary.

\begin{corollary}
	If $\ka T= \{\ka T^\flat, \ka T^ \dag, \ka T^\ddag, \ka T^\sharp\}$, then for every $k\geqslant 1$,  every $T=(A,B,C) \in \ka T_k$ and every pair of strings $v,w\in A\ominus B$, the numbers of occurrences of $0$'s and $1$'s in the two strings differ by at most $4$.
\end{corollary}
\subsection{Heat maps} 

If $T=(A,B,C)$ is a nice triple, then $A\ominus B$ is a 2-neighborly list of strings in $S^d$. The data contained in $A\ominus B $ can be encoded in an $n\times d$ array $M=[m_{ij}]$ where $n=|A\ominus B|$. We can create a \emph{heat map} for such $M$. Such a map is a rectangle dissected into $n\times d$ equal, colored boxes. We shall use three colors: red for $0$, grey for $*$ and  black for $1$. We include several heat maps related to the codes  $A\ominus B$, extracted from triples $T=(A,B,C)$ belonging  to the members of the sequence discussed in Proposition  \ref{na}.

\begin{figure}[h]
	
	\begin{subfigure}{0.45\textwidth}
		\includegraphics[width=1\linewidth, height=7cm]{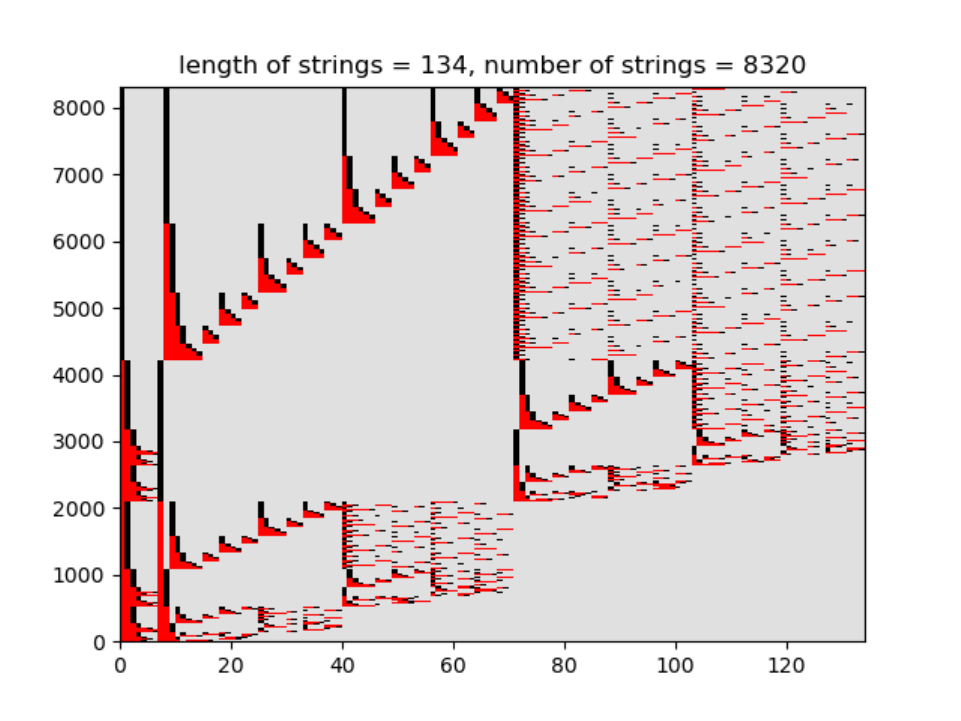}
		\caption{The heat map 1.}
		\label{Heat 1}
	\end{subfigure}
	\begin{subfigure}{0.45\textwidth}
		\includegraphics[width=1\linewidth, height=7cm]{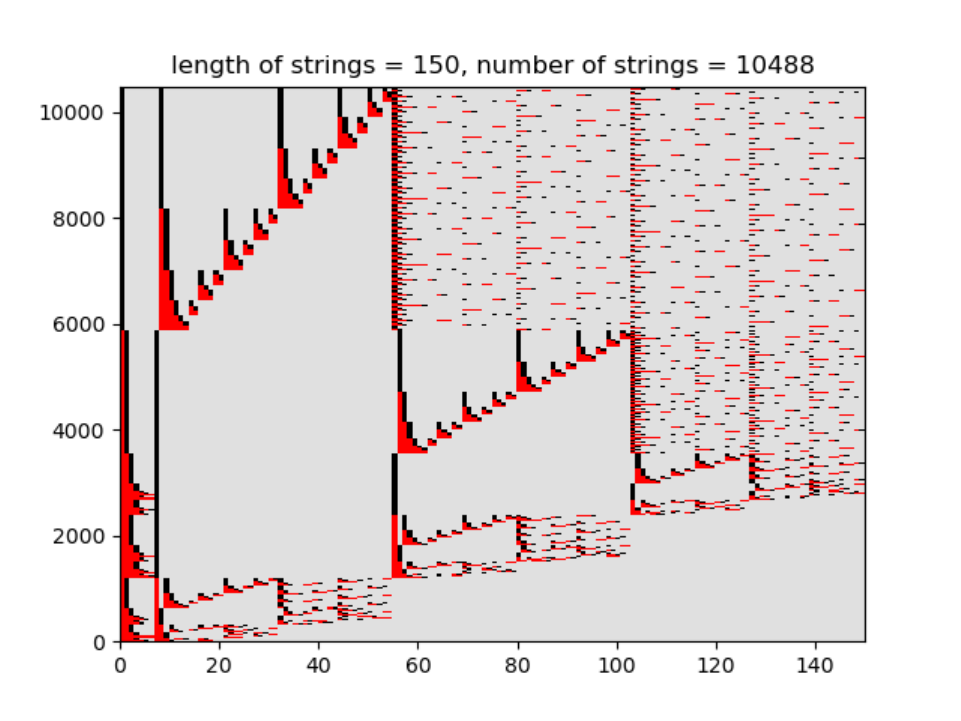}
	\caption{The heat map 2.}
	\label{Heat 2}
\end{subfigure}
		\begin{subfigure}{0.45\textwidth}
		\includegraphics[width=1\linewidth, height=7cm]{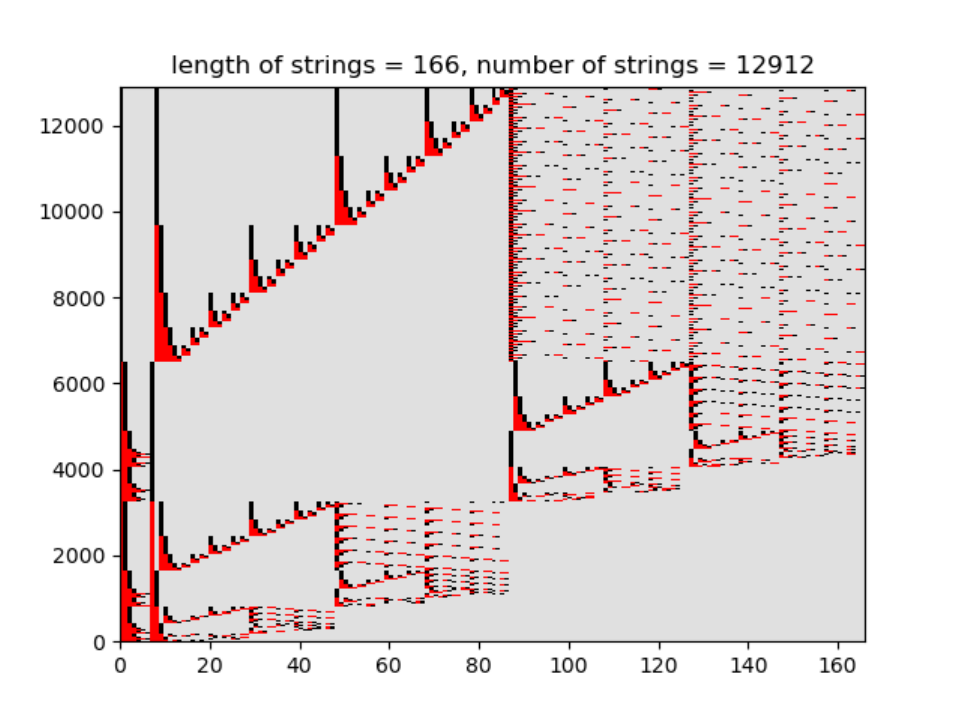}
		\caption{The heat map 3.}
		\label{Heat 3}
	\end{subfigure}
	\begin{subfigure}{0.45\textwidth}
		\includegraphics[width=1\linewidth, height=7cm]{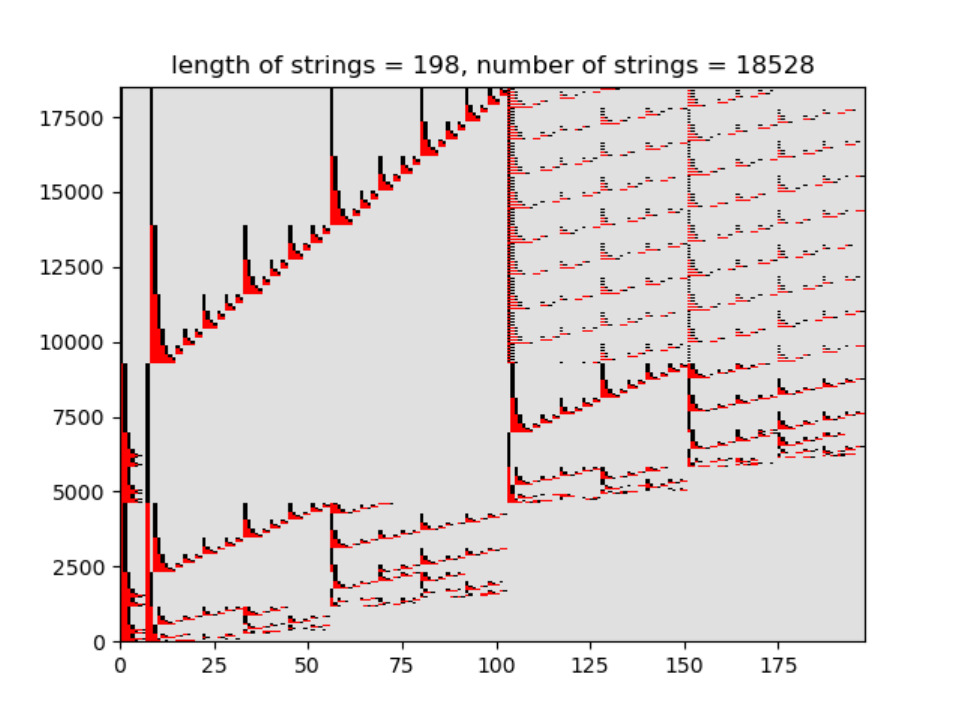}
		\caption{The heat map 4.}
		\label{Heat 4}
	\end{subfigure}
\caption{Heat maps.}
\end{figure}

\section{Final remarks}

Let us conclude with another open problem concerning a possible construction of optimal neighborly codes. It stems from Conjecture 1 in \cite{AlonGKP}, which is formulated below in a more dynamic way.

Let $v\in S^d$ be any string with at least one joker. Suppose that $v_i=*$. A \emph{splitting} of the string $v$ \emph{at position} $i$ is the pair of strings $v',v''$ obtained by replacing this joker with $0$ and $1$, respectively. For example, if $v=0**1$ and $i=3$, then $v'=0*01, v''=0*11$ is a splitting of $v$ at position $i=3$.

Consider now the following one-player game. Let $V$ be any $k$-neighborly code in $S^d$. Pick any string $v\in V$ with at least one joker. Then chose any joker in $v$ and make a splitting of $v$ at position it occupies. Now delete $v$ from the code $V$ and replace it with two strings $v',v''$ produced by the splitting. The move is \emph{legal} if the new set of strings is still a $k$-neighborly code. A natural question is:

\begin{center}
	\emph{How long one can play the splitting game starting with all-jokers string?}
\end{center}

Let us denote by $\score(k,d)$ the \emph{score} of the splitting game, i.e., the maximum size of a $k$-neighborly code obtained in the course of the splitting game in $S^d$. Clearly, $\score(k,d)\leqslant n(k,d)$, but is it possible that the most intelligent play always results in an optimal code?

Consider for example the case of $k=2$ and $d=3$ (see Figure \ref{Splitting}). After splitting the initial string $***$ at $i=1$ we get two strings, $0**$ and $1**$. Splitting the first one at $i=2$ gives $00*$ and $01*$, which together with $1**$ form a $1$-neighborly code. In the next three steps we get a $2$-neighborly code with six elements, which is best possible (since $n(2,3)=6$).
\begin{figure}[ht]
	\begin{center}
		
		\resizebox{12cm}{!}{
			
			\includegraphics{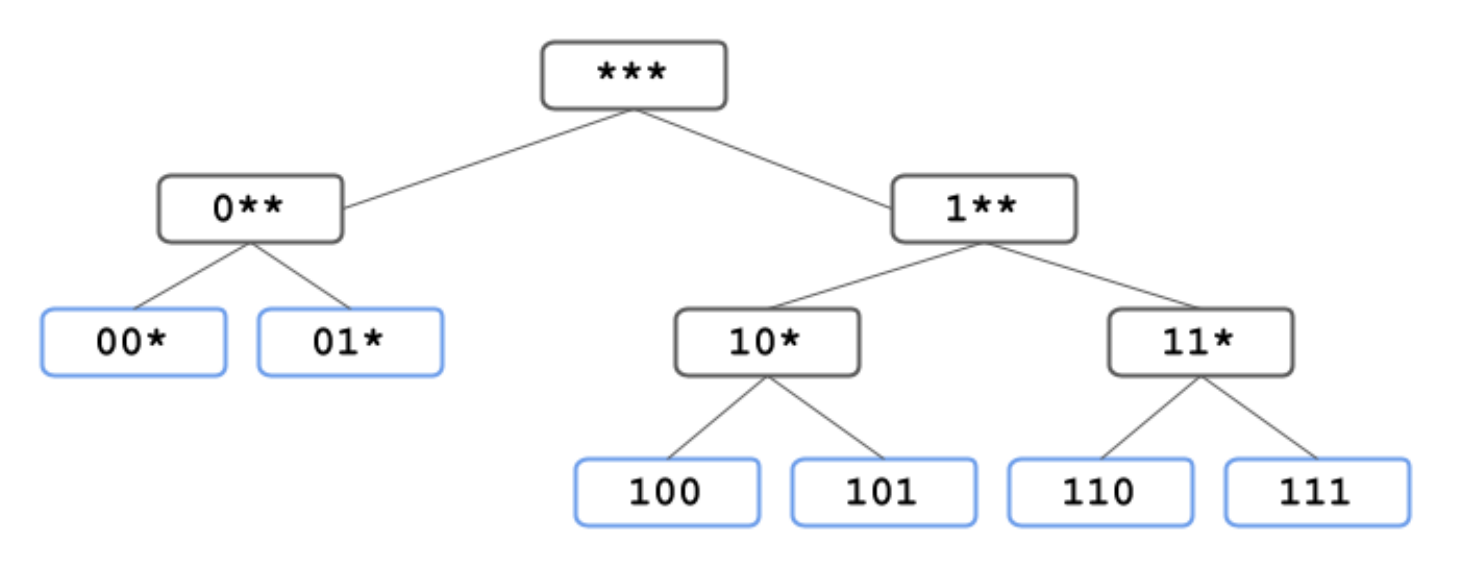}
			
		}
		\caption{The splitting game producing a maximal $2$-neighborly code (in blue boxes).}
		\label{Splitting}
	\end{center}
\end{figure}

We believe that the following conjecture is true.
\begin{conjecture}
	There exists a strategy in the splitting game resulting in a $k$-neighborly code of maximum size, that is, $\score(k,d)=n(k,d)$, for every $1\leqslant k\leqslant d$.
\end{conjecture}


\begin{thebibliography}{99}
	
	\bibitem{Aigner Ziegler}M. Aigner, G. Ziegler, Proofs from the Book, Sixth Edition, Springer Verlag, Berlin-Heidelberg, 2018.
	
	\bibitem{Alon} N. Alon,  Neighborly families of boxes and bipartite coverings, In: R. L. Graham at al. (eds.), \textit{The Mathematics of Paul Erd\"os II}, pp 27--31, Springer-Verlag, Berlin Heidelberg, 1997.
	
	\bibitem{Alon 2} N. Alon, Decomposition of the complete $r$-graph into complete $r$-partite $r$-graphs, \emph{Graphs Combin.} \textbf{2} (1986) 95--100.
	
	\bibitem{AlonGKP} N. Alon, J. Grytczuk, A.P. Kisielewicz and K. Przesławski, New bounds on the maximum number of neighborly boxes in $\mathbb{R}^d$, \textit{European J. Combin.} \textbf{114} (2023) 103797.
	
	\bibitem{ChengWXY} X. Cheng, M. Wang, Z. Xu, C. H. Yip, Exact values and improved bounds on $k$-neighborly families of boxes, \emph{European J. Combin.} $\mathbf{118}$ (2024) 103926.
	
	\bibitem{Cioaba} S. M. Cioabă, A. K\"{u}ngden, J. Verstra\"{a}te, On decompositions of complete hypergraphs, \emph{J. Combin. Theory Ser. A} \textbf{116} (2009) 1232--1234.
	
	\bibitem{GP} R. L. Graham and H. O. Pollak, On embedding graphs in squashed cubes, In: \textit{Lecture Notes in Mathematics} \textbf{303}, pp 99--110, Springer Verlag, New York-Berlin-Heidelberg, 1972.
	
	\bibitem{HS}  H. Huang and B. Sudakov, A counterexample to Alon-Saks-Seymour conjecture and related problems, \textit{Combinatorica} \textbf{32} (2012) 205--219.
	
	\bibitem{Luba} S. Łuba, Systems of Unit Cubes in $\mathbb{R}^d$, Master thesis, Uniwersytet Zielonogórski, 2023, (in Polish).
	
	\bibitem{Peck} G. W. Peck, A new proof of a theorem of Graham and Pollak, \textit{Discrete Math.} \textbf{49} (1984) 327--328.
	
	\bibitem{Tverberg} H. Tverberg, On the decomposition of $K_n$ into complete bipartite graphs, \textit{J. Graph Theory} \textbf{6} (1982) 493--494.
	
	\bibitem{van Lint} J.H. van Lint, $\{0,1,\ast\}$ distance problems in combinatorics, in: Surveys in Combinatorics 1985 (Glasgow, 1985), in: London Math. Soc. Lecture Note Ser., vol. 103, Cambridge Univ. Press, Cambridge, 1985, pp. 113--135.
	
	\bibitem{Vishwanathan1} S. Vishwanathan, A polynomial space proof of the Graham–Pollak theorem, \emph{J. Combin. Theory Ser. A} \textbf{115} (2008) 674--676.
	
	\bibitem{Vishwnathan2} S. Vishwanathan, A counting proof of the Graham–Pollak Theorem, \emph{Discrete Math.} \textbf{313} (2013) 765--766.
	
	\bibitem{Zaks3} J. Zaks, How Does a Complete Graph Split into Bipartite Graphs and How Are Neighborly Cubes Arranged?, \textit{Amer. Math. Monthly}  \textbf{92} (1985) 568--571.
	
\end{thebibliography}
\end{document}